\documentclass[12pt,reqno]{amsart} 
\usepackage{mathtools}

\mathtoolsset{showonlyrefs}
\usepackage[hmarginratio={1:1},scale=0.8]{geometry} 
\usepackage{amsaddr} 
\usepackage{enumerate} 
\usepackage{amssymb} 
\usepackage{cite} 
\usepackage{bm} \numberwithin{equation}{section}

\DeclareMathOperator{\Span}{span}
\DeclareMathOperator{\supt}{supp}
\newcommand{\ko}{\mathcal K_0^n}

\newcommand{\ke}{\mathcal K_e^n}
\newcommand{\rn}{\mathbb R^n}

\newcommand{\sn}{ {S^{n-1}}}

\newtheorem*{problem}{Problem}
\newtheorem{lemma}{Lemma}[section] 
\newtheorem{theorem}[lemma]{Theorem}

\title{The $L_p$ Aleksandrov problem for origin-symmetric polytopes} 
\author[Y. Zhao]{Yiming Zhao}
\address{Department of Math and Computer Science,  St. John's University, 8000 Utopia Parkway, Queens, NY 11439}
\email{yiming.zhao.math@gmail.com}

\begin{document} 
\maketitle 

\begin{abstract}
The $L_p$ Aleksandrov integral curvature and its corresponding characterization problem---the $L_p$ Aleksandrov problem---were recently introduced by Huang, Lutwak, Yang, and Zhang. The current work presents a solution to the $L_p$ Aleksandrov problem for origin-symmetric polytopes when $-1<p<0$.
\end{abstract}

\section{Introduction}

The classical Aleksandrov problem is the counterpart of the Minkowski problem---a fundamental problem in the Brunn-Minkowski theory whose influence reaches many fields of mathematics including convex geometry, differential geometry, PDE, and functional analysis. The Aleksandrov problem is the measure characterization problem for Aleksandrov integral curvature $J(K,\cdot)$ (also known as integral Gauss curvature), the most studied curvature measure which was defined by Aleksandrov \cite{MR0007625}. When the convex body $K$ is sufficiently smooth, the Aleksandrov integral curvature of $K$ (when viewed as a measure on $\partial K$) has the Gauss curvature as its density. 

The Aleksandrov problem was completely solved by Aleksandrov himself using a topological argument, see Aleksandrov \cite{MR0007625}. Alternative approaches that connect the Aleksandrov problem to optimal mass transport were given by Oliker \cite{MR2332603} and more recently by Bertrand \cite{MR3523119}. 

The last three decades saw the rapid and flourish development of the $L_p$ Brunn-Minkowski theory that was initiated by Firey but only truly gained life when Lutwak \cite{MR1231704,MR1378681} began to systematically work on it in the early 1990s. The $L_p$ Brunn-Minkowski theory is arguably the most vibrant theory in modern convex geometry and has been the breeding ground for many important results. 
The $L_p$ Minkowski problem is the fundamental problem in the $L_p$ Brunn-Minkowski theory and characterizes the $L_p$ surface area measure that sits in the center of the theory. In particular, the discovery of an important class of affine isoperimetric inequalities---the sharp affine $L_p$ Sobolev inequality---owes to the solution of the $L_p$ Minkowski problem for $p\geq 1$, see \cite{MR1987375}. This effort has over the years inspired many more sharp affine isoperimetric inequalities, see \cite{MR1863023, MR1987375, MR2530600}.

The corresponding measure characterization problem (geometric measure, resp.) for the Aleksandrov problem (Aleksandrov integral curvature, resp.) in the $L_p$ Brunn-Minkowski theory had been long sought-for. In a recent groundbreaking work \cite{HLYZ}, Huang, Lutwak, Yang \& Zhang (Huang-LYZ) discovered that Aleksandrov integral curvature naturally arises as the ``differential'' of a certain entropy integral. Following their work, they introduced the $L_p$ Aleksandrov integral curvature in \cite{HLYZ2} and posed the measure characterization problem called the $L_p$ Aleksandrov problem. More details will follow shortly.

The purpose of the current work is to solve the $L_p$ Aleksandrov problem when $-1<p<0$ in the case of origin-symmetric polytopes.

We shall provide some background on the $L_p$ Brunn-Minkowski theory. 

The $L_p$ surface area measure $S_p(K,\cdot)$ is the fundamental geometric measure in the $L_p$ Brunn-Minkowski theory. In fact, many key concepts in the $L_p$ theory including the $L_p$ mixed volume and the $L_p$ affine surface area can be defined solely using the $L_p$ surface area measure.

The $L_p$ Minkowski problem asks: given a Borel measure $\mu$ on $S^{n-1}$, what are necessary and sufficient conditions on $\mu$ so that there exists a convex body $K$ such that $\mu$ is exactly the $L_p$ surface area measure of $K$? When $p=1$, the $L_p$ Minkowski problem is the same as the classical Minkowski problem which was solved by Minkowski, Fenchel \& Jessen, Aleksandrov, etc. Regularity results on the Minkowski problem include the influential paper \cite{MR0423267} by Cheng \& Yau. The solution, when $p>1$, was given by Lutwak \cite{MR1231704} when $\mu$ is an even measure and Chou \& Wang \cite{MR2254308} for arbitrary $\mu$. See also Chen \cite{MR2204749}, Chen, Li \& Zhu \cite{MR3680945}, Huang, Liu \& Xu \cite{MR3366857}, Hug-LYZ \cite{MR2132298}, Jian, Lu \& Wang \cite{MR3366854}, Lutwak \& Oliker \cite{MR1316557}, LYZ \cite{MR2067123}, , and Zhu \cite{MR3352764,guangxianindiana}. 

The $L_p$ Minkowski problem contains two major unsolved cases. 

When $p=-n$, the $L_{-n}$ surface area measure $S_{-n}(K,\cdot)$ is also known as the centro-affine surface area measure whose density in the smooth case is the centro-affine Gauss curvature. The characterization problem, in this case, is the centro-affine Minkowski problem posed in Chou \& Wang \cite{MR2254308}. See also Jian, Lu \& Zhu \cite{MR3479715}, Lu \& Wang \cite{MR2997361}, Zhu \cite{MR3356071}, etc., on this problem.

When $p=0$, the $L_0$ surface area measure $S_{0}(K,\cdot)$ is the cone volume measure whose total measure is the volume of $K$. Cone volume measure is the only one among all $L_p$ surface area measures that is SL$(n)$ invariant. It is still being intensively studied, see, for example, Barthe, Gu\'{e}don, Mendelson \& Naor \cite{MR2123199}, B\"{o}r\"{o}czky \& Henk \cite{MR3415694}, B\"{o}r\"{o}czky-LYZ \cite{MR2964630,BLYZ,MR3316972}, Henk \& Linke \cite{MR3148545}, Ludwig \cite{MR2652209}, Ludwig \& Reitzner \cite{MR2680490}, Naor \cite{MR2262841}, Naor \& Romik \cite{MR1962135}, Paouris \& Werner \cite{MR2880241}, Stancu \cite{MR1901250,MR2019226}, Xiong \cite{MR2729006}, Zhu \cite{MR3228445}, and Zou \& Xiong \cite{MR3255458}. The characterization problem for the cone volume measure is the logarithmic Minkowski problem. A complete solution to the existence part of the logarithmic Minkowski problem, when restricting to even measures and the class of origin-symmetric convex bodies, was given by B\"{o}r\"{o}czky-LYZ \cite{BLYZ}. In the general case (non-even case), different efforts have been made by B\"{o}r\"{o}czky, Heged\H{u}s \& Zhu \cite{Boroczky20062015}, Stancu \cite{MR1901250,MR2019226}, Zhu \cite{MR3228445}, and most recently by Chen, Li \& Zhu \cite{CLZ}. The logarithmic Minkowski problem has strong connections with isotropic measures (B{\"o}r{\"o}czky-LYZ \cite{MR3316972}) and curvature flows (Andrews \cite{MR1714339,MR1949167}).

In a groundbreaking work \cite{HLYZ}, Huang-LYZ discovered a new family of geometric measures called dual curvature measures $\widetilde{C}_q(K,\cdot)$ and the variational formula that leads to them. The dual Minkowski problem---the problem of prescribing dual curvature measures---was posed as well. The dual Minkowski problem miraculously contains problems such as the Aleksandrov problem ($q=0$) and the logarithmic Minkowski problem ($q=n$) as special cases. The problem quickly became the center of attention, see, for example, \cite{BF,BHP, BLYZZ,MR3725875,HLYZ2, HZ,GHWXY,LYZpDC,LSW,Zha1, Zha2}. 

The variational formula for Aleksandrov's integral curvature obtained in \cite{HLYZ} allowed the following discovery, see \cite{HLYZ2}: for each $0\neq p\in \mathbb{R}$ and $K\in \mathcal{K}_o^n$, define the $L_p$ Aleksandrov integral curvature, $J_p(K,\cdot)$, of $K$ as the unique Borel measure on $S^{n-1}$ such that 
\begin{equation}
\left.\frac{d}{dt}\mathcal{E}(K\hat{+}_p t\cdot Q)\right|_{t=0}=\frac{1}{p} \int_{S^{n-1}}\rho_Q(u)^{-p}dJ_p(K,u)
\end{equation}
holds for every $Q\in \ko$, where $\mathcal{E}(\cdot)$ is the entropy integral defined by 
\begin{equation}
\label{eq definition of E}
\mathcal{E}(Q)=-\int_{S^{n-1}}\log h_Q(v)dv,
\end{equation} 
and $K\hat{+}_p t\cdot Q$ is the harmonic $L_p$-combination defined by 
\begin{equation}
K\hat{+}_p t\cdot Q=\left(K^*+_p t\cdot Q^*\right)^*.
\end{equation}
Here $K^*$ is the polar body of $K$.

The $L_p$ Aleksandrov integral curvature is absolutely continuous with respect to the classical Aleksandrov integral curvature $J(K,\cdot)$:
\begin{equation}
dJ_p(K,\cdot)=\rho_K^pdJ(K,\cdot).
\end{equation}
Hence $J_p(K,\cdot)$ is defined for $p=0$ and $J_0(K,\cdot)=J(K,\cdot)$.

The $L_p$ Aleksandrov problem is the measure characterization problem for the $L_p$ Aleksandrov integral curvature.
\begin{problem}[The $L_p$ Aleksandrov problem]
Given a non-zero finite Borel measure $\mu$ on $S^{n-1}$ and $p\in \mathbb{R}$. What are the necessary and sufficient conditions on $\mu$ so that there exists $K\in \mathcal{K}_o^n$ such that $\mu = J_p(K,\cdot)$?
\end{problem}

When the given measure $\mu$ has a density $f$, solving the $L_p$ Aleksandrov problem is the same as solving the following Monge-Amp\`{e}re type equation on $S^{n-1}$:
\begin{equation}
h^{1-p} (|\nabla h|^2+h^2)^{-\frac{n}{2}}\det(\nabla^2 h+hI)=f,
\end{equation}
where $h$ is the unknown, $\nabla h$ and $\nabla^2 h$ are the gradient and Hessian of $h$ on $S^{n-1}$ with respect to an orthonormal basis, and $I$ is the identity matrix.

Huang-LYZ established the existence part of the problem in several situations. When $p>0$, the existence part is completely established.
\begin{theorem}[\cite{HLYZ2}]
\label{thm 1}
Suppose $p\in (0,\infty)$ and $\mu$ is a non-zero finite Borel measure on $S^{n-1}$. There exists $K\in \mathcal{K}_o^n$ such that $\mu = J_p(K,\cdot)$ if and only if $\mu$ is not concentrated in any closed hemisphere.
\end{theorem}

The case $p<0$ is much more complicated. The existence part of the $L_p$ Aleksandrov problem was only able to be established under very strong assumptions.
\begin{theorem}[\cite{HLYZ2}]
\label{HLYZ theorem}
Suppose $p\in (-\infty,0)$ and $\mu$ is a non-zero even finite Borel measure on $S^{n-1}$. If $\mu$ vanishes on all great subspheres of $S^{n-1}$, then there exists $K\in \mathcal{K}_o^n$ such that $\mu=J_p(K,\cdot)$.
\end{theorem}

Note that the conditions in Theorem \ref{HLYZ theorem} are quite strong. In particular, an important class of convex bodies---polytopes---are not included in the solution, for the simple reason that the $L_p$ Aleksandrov integral curvature of a polytope must be discrete and therefore must obtain positive measure on many great subspheres. In fact, the classical Minkowski problem and some cases of the $L_p$ Minkowski problem were first solved for the polytopal case and then solved for the general case using approximation.

Theorem \ref{HLYZ theorem} was shown using variational method. What makes the problem especially challenging in the case when $\mu$ has a positive concentration in one of the proper subspaces is the behavior of the functional $\Phi$ in the maximization problem \eqref{eq variational problem}. See also \eqref{eq local 1002}. When the given measure $\mu$ has even the slightest concentration in a proper subspace $u^\perp$, the functional $\Phi$ will still obtain a finite number for any convex body in $u^\perp$. This feature of $\Phi$ made it extremely challenging to show that the final solution $K$ does not collapse into a lower dimensional subspace. 

The aim of the current work is to show that the $L_p$ Aleksandrov problem has a solution when $-1<p<0$ and the given measure $\mu$ is an even discrete measure.
\begin{theorem}
\label{main theorem}
Suppose $p\in (-1,0)$ and $\mu$ is a non-zero, even, discrete, finite, Borel measure on $S^{n-1}$. There exists an origin-symmetric polytope $K\in \mathcal{K}_e^n$ such that $\mu=J_p(K,\cdot)$ if and only if $\mu$ is not concentrated entirely on any great subspheres.
\end{theorem}
Note that although still variational in nature, the approach here is vastly different from that in \cite{HLYZ2}. It should also be pointed out that one might use the solution obtained in the current work to obtain a solution to the even $L_p$ Aleksandrov problem for $-1<p<0$ via approximation.

\section{Preliminaries}This section is divided into two subsections. In the first subsection, basics in the theory of convex bodies will be covered. In the second subsection, the notion of $L_p$ Aleksandrov integral curvature and $L_p$ Aleksandrov problem will be introduced.

\subsection{Basics in the theory of convex bodies} The book \cite{schneider2014} by Schneider offers a comprehensive overview of the theory of convex bodies.

Let $\rn$ be the $n$-dimensional Euclidean space. The unit sphere in $\rn$ is denoted by $\sn$ and the volume of the unit ball will be written as $\omega_n$. A convex body in $\rn$ is a compact convex set with nonempty interior. The boundary of $K$ is written as $\partial K$. Denote by $\ko$ the class of convex bodies that contain the origin in their interiors in $\rn$ and by $\ke$ the class of origin-symmetric convex bodies in $\rn$.

Let $K$ be a compact convex subset of $\rn$. The support function $h_K$ of $K$ is defined by
\[ h_K(y) = \max\{x\cdot y : x\in K\}, \quad y\in\rn. \]
The support function $h_K$ is a continuous function homogeneous of degree 1. Suppose $K$ contains the origin in its interior. The radial function $\rho_K$ is defined by
\[ \rho_K(x) = \max\{\lambda : \lambda x \in K\}, \quad x\in \rn\setminus \{0\}. \]
The radial function $\rho_K$ is a continuous function homogeneous of degree $-1$. It is not hard to see that $\rho_K(u)u \in
\partial K$ for all $u\in S^{n-1}$.

For a convex body $K\in \ko$, the polar body of $K$ is given by
\begin{equation*}
K^* = \{y\in \rn: y\cdot x\leq 1, \text{ for all } x\in K\}.
\end{equation*}
It is simple to check that $K^*\in \ko$ and that
\begin{equation}
\label{eq polar body support function}
\begin{aligned}
h_{K^*}(x) &= 1/\rho_K(x),\\
\rho_{K^*}(x) &= 1/h_K(x),
\end{aligned}
\end{equation}
for $x\in \rn \setminus \{o\}$. Moreover, we have $(K^*)^*=K$.

For each $f\in C^+(\sn)$, the Wulff shape $\bm{[}f\bm{]}$ generated by $f$ is the convex body defined by
\begin{equation*}
\bm{[}f\bm{]}= \{x\in \rn: x\cdot v \leq f(v), \text{ for all }v \in \sn\}.
\end{equation*}
It is apparent that $h_{\bm{[}f\bm{]}}\leq f$ and $\bm{[}h_K\bm{]} = K$ for each $K\in \ko$.

The $L_p$ combination of two convex bodies $K, L \in \ko$ was first studied by Firey and was the starting point of the now rich $L_p$ Brunn-Minkowski theory developed by Lutwak~\cite{MR1231704,MR1378681}. For $t,s>0$, the $L_p$ combination of $K$ and $L$, denoted by $t\cdot K+_{\mkern-2.7mu p} s\cdot L$, is defined to be the Wulff shape generated by the function $h_{t,s}$ where
\begin{equation*}
h_{t,s}=
\begin{cases}
\left(th_K^p+sh_L^p\right)^\frac{1}{p}, & \text{ if } p\neq 0,\\
h_K^th_L^s, & \text{ if } p=0.
\end{cases}
\end{equation*}
When $p\geq 1$, by the convexity of $\ell_p$ norm, we get that
\begin{equation*}
h_{K+_{\mkern-2.7mu p}t\cdot L}^p = h_K^p+th_L^p.
\end{equation*}
Define the $L_p$ harmonic combination $t\cdot K \hat{+}_{\mkern-2.7mu p} s\cdot L$ by 
\begin{equation}
t\cdot K \hat{+}_{\mkern-2.7mu p} s\cdot L = \left( t\cdot K^*{+}_{\mkern-2.7mu p} \,s\cdot L^*\right)^*.
\end{equation}

Suppose $K_i$ is a sequence of convex bodies in $\rn$. We say $K_i$ converges to a compact convex subset $K\subset \rn$ if
\begin{equation}
\label{eq convergence convex bodies}
\max\{|h_{K_i}(v)-h_K(v)|:v\in S^{n-1}\}\rightarrow 0,
\end{equation}
as $i\rightarrow \infty$. If $K$ contains the origin in its interior, equation \eqref{eq convergence convex bodies} implies
\begin{equation*}
\max\{|\rho_{K_i}(u)-\rho_K(u)|:u\in S^{n-1}\}\rightarrow 0,
\end{equation*}
as $i\rightarrow \infty$.

For a compact convex subset $K$ in $\rn$ and $v \in \sn$, the supporting hyperplane $H(K,v)$ of $K$ at $v$ is given by
\begin{equation*}
H(K,v)=\{x\in K: x\cdot v = h_K(v)\}.
\end{equation*}
By its definition, the supporting hyperplane $H(K,v)$ is non-empty and contains only boundary points of $K$. For $x\in H(K,v)$, we say $v$ is an outer unit normal of $K$ at $x\in \partial K$.

Let $\omega\subset \sn$ be a Borel set. The radial Gauss image of $K$ at $\omega$, denoted by $\bm{\alpha}_K(\omega)$, is defined to be the set of all unit vectors $v$ such that $v$ is an outer unit normal of $K$ at some boundary point $u\rho_K(u)$ where $u\in \omega$, i.e.,
\begin{equation*}
\bm{\alpha}_K(\omega) = \{v \in S^{n-1}: v\cdot u\rho_K(u) = h_K(v) \text{ for some }u\in \omega\}.
\end{equation*}
When $\omega = \{u\}$ is a singleton, we usually write $\bm{\alpha}_K(u)$ instead of the more cumbersome notation $\bm{\alpha}_K(\{u\})$. Let $\omega_K$ be the subset of $\sn$ such that $\bm{\alpha}_K(u)$ contains more than one element for each $u\in \omega_K$. By Theorem 2.2.5 in \cite{schneider2014}, the set $\omega_K$ has spherical Lebesgue measure 0. The radial Gauss map of $K$, denoted by $\alpha_K$, is the map defined on $\sn\setminus \omega_K$ that takes each point $u$ in its domain to the unique vector in $\bm{\alpha}_K(u)$. Hence $\alpha_K$ is defined almost everywhere on $\sn$ with respect to the spherical Lebesgue measure.

Let $\eta \subset \sn$ be a Borel set. The reverse radial Gauss image of $K$, denoted by $\bm{\alpha}_K^*(\eta)$, is defined to be the set of all radial directions such that the corresponding boundary points have at least one outer unit normal in $\eta$, i.e.,
\begin{equation*}
\bm{\alpha}_K^*(\eta) = \{u\in \sn: v\cdot u\rho_K(u) = h_K(v) \text{ for some } v \in \eta\}.
\end{equation*}
When $\eta = \{v\}$ is a singleton, we usually write $\bm{\alpha}_K^*(v)$ instead of the more cumbersome notation $\bm{\alpha}_K^*(\{v\})$. Let $\eta_K$ be the subset of $\sn$ such that $\bm{\alpha}_K^*(v)$ contains more than one element for each $v\in \eta_K$. By Theorem 2.2.11 in \cite{schneider2014}, the set $\eta_K$ has spherical Lebesgue measure 0. The reverse radial Gauss map of $K$, denoted by $\alpha_K^*$, is the map defined on $\sn \setminus \eta_K$ that takes each point $v$ in its domain to the unique vector in $\bm{\alpha}_K^*(v)$. Hence $\alpha_K^*$ is defined almost everywhere on $\sn$ with respect to the spherical Lebesgue measure.

\subsection{$L_p$ Aleksandrov integral curvature and the $L_p$ Aleksandrov problem}

For $K\in \ko$, the Aleksandrov integral curvature of $K$, denoted by $J(K,\cdot)$, is a Borel measure on $\sn$ given by 
\begin{equation}
\label{eq local 2000}
J(K,\omega) = \mathcal{H}^{n-1}(\bm{\alpha}_K(\omega)).
\end{equation}
It is apparent that $|J(K,\cdot)| = n\omega_n$. The classical Aleksandrov problem is the measure characterization problem for Aleksandrov integral curvature: \textit{given a Borel measure $\mu$ on $\sn$ with $|\mu|=n\omega_n$, under what conditions on $\mu$ is there a convex body $K\in \ko$ such that $\mu = J(K,\cdot)$?}

The Aleksandrov problem was completely solved by Aleksandrov himself \cite{MR0007625} using his mapping lemma. In particular, there exists a $K\in \ko$ with $\mu=J(K,\cdot)$ if and only if the given measure $\mu$ satisfies the following Aleksandrov condition:
\begin{equation}
\mu(\omega)<\mathcal{H}^{n-1}(\sn \setminus \omega^*),
\end{equation}
for each non-empty spherically convex $\omega\subset \sn$. Here $\omega^*$ is the given by
\begin{equation*}
\omega^*=\{v\in \sn: v\cdot u\leq 0, \forall u\in \omega\}.
\end{equation*} 
Moreover, the convex body $K$, if it exists, is unique up to a dilation. Different approaches to the Aleksandrov problem via optimal mass transport was given by Oliker \cite{MR2332603} and more recently by Bertrand \cite{MR3523119}.

Following the groundbreaking work \cite{HLYZ} by Huang-LYZ, the four authors discovered in \cite{HLYZ2} that Aleksandrov integral curvature arises naturally by ``differentiating'' the entropy integral $\mathcal{E}$ given in \eqref{eq definition of E}. In particular, 
\begin{equation}
\label{eq local 2001}
\left.\frac{d}{dt}\mathcal{E}(K\hat{+}_{\mkern-2.7mu o}\, t\cdot Q)\right|_{t=0}=-\int_{S^{n-1}}\log \rho_Q(u)dJ(K,u)
\end{equation}
holds for each $Q\in \ko$. Note that, instead of defining Aleksandrov integral curvature as in \eqref{eq local 2000}, one may define $J(K,\cdot)$ as the unique Borel measure on $\sn$ such that \eqref{eq local 2001} holds for each $Q\in \ko$. Inspired by the success of the $L_p$ Brunn-Minkowski theory over the last three decades, Huang-LYZ \cite{HLYZ2} discovered a new family of geometric measures, called the $L_p$ Aleksandrov integral curvature. For $K\in \ko$ and $p\neq 0$, the $L_p$ Aleksandrov integral curvature of $K$, $J_p(K,\cdot)$, is defined to be the unique Borel measure on $\sn$ such that 
\begin{equation}
\left.\frac{d}{dt}\mathcal{E}(K\hat{+}_{\mkern-2.7mu p} \,t\cdot Q)\right|_{t=0}=\frac{1}{p} \int_{S^{n-1}}\rho_Q(u)^{-p}dJ_p(K,u)
\end{equation}
holds for each $Q\in \ko$.

The $L_p$ Aleksandrov integral curvature has the following integral representation:
\begin{equation}
J_p(K,\cdot )=\rho_K^p dJ(K,\cdot).
\end{equation}
Hence, we may define $J_0(K,\cdot)$ as the classical Aleksandrov integral curvature $J(K,\cdot)$.

When the body $K$ is sufficiently smooth, the $L_p$ Aleksandrov integral curvature $J(K,\cdot)$ is absolute continuous with respect to the spherical Lebesgue measure and its Radon-Nikodym derivative is given by  
\begin{equation}
\label{eq local 2002}
h^{1-p}(|\nabla h|^2+h^2)^{-\frac{n}{2}}\det (\nabla^2 h+hI),
\end{equation}
where $\nabla h$ and $\nabla^2 h$ are the gradient and Hessian of $h$ on $\sn$ with respect to an orthonormal basis.

Huang-LYZ \cite{HLYZ2} posed the following $L_p$ Aleksandrov problem.
\begin{problem}[The $L_p$ Aleksandrov problem]
Given a non-zero finite Borel measure $\mu$ on $S^{n-1}$ and $p\in \mathbb{R}$. What are the necessary and sufficient conditions on $\mu$ so that there exists $K\in \mathcal{K}_o^n$ such that $\mu = J_p(K,\cdot)$?
\end{problem}

By \eqref{eq local 2002}, the $L_p$ Aleksandrov problem reduces to the following PDE when the given measure $\mu$ has a density $f$:
\begin{equation}
h^{1-p} (|\nabla h|^2+h^2)^{-\frac{n}{2}}\det(\nabla^2 h+hI)=f.
\end{equation}

When $p=0$, the $L_0$ Aleksandrov problem is nothing but the classical Aleksandrov problem. When $p>0$, the existence  part of the $L_p$ Aleksandrov problem was completely settled in Huang-LYZ \cite{HLYZ2}, see Theorem \ref{thm 1}. However, when $p<0$, a relatively strong condition was required in Huang-LYZ \cite{HLYZ2} to show the existence, see Theorem \ref{HLYZ theorem}. This condition excludes a very important subclass of convex bodies---polytopes. It is the aim of the current work to fill that gap in the case when $-1<p<0$ and the polytope is origin-symmetric. 

The proof adopted here is variational in nature. In Section \ref{sec optimization problem}, we shall convert the $L_p$ Aleksandrov problem, when the given measure is discrete and even, to an optimization problem. In Section \ref{sec solution}, the proposed optimization problem will be solved. The proof to the main theorem, Theorem \ref{main theorem}, is given at the end of Section \ref{sec solution}.
\section{Optimization problem}
\label{sec optimization problem}

The following lemma was given in Huang-LYZ \cite{HLYZ2}, which connects the $L_p$ Aleksandrov problem to an optimization problem.
\begin{lemma}[Lemma 5.3, \cite{HLYZ2}]
\label{lemma in HLYZ2}
Suppose $p\neq 0$. Let $\mu$ be a finite even, Borel measure on $\sn$ and $K\in \ke$ be a body such that 
\begin{equation*}
\Psi(K)=\sup\{\Psi(Q):Q\in \ke\},
\end{equation*}
where $\Psi(Q)=\frac{1}{n\omega_n} \mathcal{E}(Q)-\frac{1}{p}\log \int_{S^{n-1}}\rho_Q^{-p}d\mu$. Then, there exists $c>0$ such that $\mu = J_p(cK,\cdot)$.
\end{lemma}

We shall now adapt the above lemma to the discrete setting.

Suppose $\mu$ is an even discrete measure whose support is $\{\pm u_1, \pm u_2, \dots, \pm u_N\}$. Let $D_\mu\subset \mathcal{K}_o^n$ be the set of all origin-symmetric convex polytopes whose set of vertices is a subset of $\{\pm u_1, \pm u_2,\dots, \pm u_N\}$. It is obvious that if $K\in D_\mu$, then there exists $\rho_1, \dots, \rho_N>0$ such that
\begin{equation}
D_\mu=\text{conv}\{\pm \rho_1u_1, \pm \rho_2 u_2,\dots, \pm \rho_Nu_N\}.
\end{equation}

The following lemma converts the even discrete $L_p$ Aleksandrov problem into a maximization problem.
\begin{lemma}
\label{lemma optimization problem}
Suppose $\mu$ is an even discrete measure and $p\neq 0$. If there exists an origin-symmetric $K\in \mathcal{K}_e^n$ such that 
\begin{equation}
\label{eq variational problem}
\Phi(K)=\sup_{Q\in D_\mu} \Phi(Q),
\end{equation}
where $\Phi:\mathcal{K}_e^n \rightarrow \mathbb{R}$ is given by 
\begin{equation}
\label{eq local 1002}
\Phi(Q) = \frac{1}{n\omega_n} \mathcal{E}(Q)-\frac{1}{p}\log \sum_{u_i\in \supt \mu} \rho_Q(u_i)^{-p}\mu(\{u_i\}),
\end{equation}
then there exists $c>0$ such that 
\begin{equation}
\mu=J_p(cK,\cdot).
\end{equation}
\end{lemma}
\begin{proof}
It is obvious that
\begin{equation}
\label{eq local 1001}
\sup_{Q\in D_\mu} \Phi(Q)\leq \sup_{Q\in \mathcal{K}_e^n} \Phi(Q).
\end{equation}
On the other side, for each $Q\in \mathcal{K}_e^n$, let 
\begin{equation}
\widetilde{Q}=\text{conv} \{\rho_{Q}(u_i) u_i: u_i\in \supt \mu\}.
\end{equation}
Then, $\rho_{\widetilde{Q}}(u_i)\geq \rho_{Q}(u_i)$ for each $u_i\in \supt \mu$ and $h_{\widetilde{Q}}(v)\leq h_{Q}(v)$ for each $v\in S^{n-1}$. This implies that
\begin{equation}
\Phi(Q)\leq \Phi(\widetilde{Q}).
\end{equation}
This, in combination with \eqref{eq local 1001}, shows that 
\begin{equation}
\Phi(K)=\sup_{Q\in D_\mu} \Phi(Q)=\sup_{Q\in \mathcal{K}_e^n} \Phi(Q).
\end{equation}

Note that when the given measure $\mu$ is discrete, $\Psi(\cdot)=\Phi(\cdot)$. According to 
Lemma \ref{lemma in HLYZ2}, there exists $c>0$ such that $\mu=J_p(cK,\cdot)$.
\end{proof}

\section{Solving the optimization problem}
\label{sec solution}
\begin{lemma}
\label{lemma convergence}
Suppose $\mu$ is an even discrete measure on $S^{n-1}$ whose support is not contained in any great subspheres. Let $Q^j \in D_\mu$ be such that $\max_{u\in S^{n-1}}\rho_{Q^j}(u)=1$. Assume there exists an origin-symmetric compact convex set $Q^0$ such that $Q^j$ converges to $Q^0$ in Hausdorff metric. Then, by possibly taking a subsequence, 
\begin{equation}
\lim_{j\rightarrow \infty}\Phi(Q^j)=\Phi(Q^0).
\end{equation}
\end{lemma}
\begin{proof}

Since $Q^j \in D_\mu$, by possibly taking a subsequence, we may assume $\rho_{Q^j}(u_{i_0})=1$ for some $u_{i_0}\in \supt \mu$. By the definition of support function and the fact that $Q^j$ is origin-symmetric,
\begin{equation}
|u_{i_0}\cdot v|\leq h_{Q^j}(v)\leq 1, \qquad v\in S^{n-1}.
\end{equation}
Hence, 
\begin{equation}
|\log h_{Q^j}(v)|\leq -\log |u_{i_0}\cdot v|, \qquad v\in S^{n-1}.
\end{equation}
Notice that $\log |u_{i_0}\cdot v|$ is an integrable function on $S^{n-1}$. Since $Q^j$ converges to $Q^0$ in Hausdorff metric, $h_{Q^j}$ converges to $h_{Q^0}$ pointwise. This,  combined with the fact that $h_{Q^0}>0$ almost everywhere (since $Q^0$ has diameter bigger than 1), implies that $\log h_{Q^j}$ converges to $\log h_{Q^0}$ almost everywhere. By dominated convergence theorem,
\begin{equation}
\label{eq local 1004}
\lim_{j\rightarrow \infty}\mathcal{E}(Q^j) = \mathcal{E}(Q^0).
\end{equation}

On the other side, since $Q^j$ converges to $Q^0$ in Hausdorff metric, we have 
\begin{equation}
\lim_{j\rightarrow \infty}\rho_{Q^j}(u_i)=\rho_{Q^0}(u_i),\qquad \forall u_i\in \supt \mu.
\end{equation}
Hence,
\begin{equation}
\lim_{j\rightarrow \infty} \sum_{u_i\in \supt \mu} \rho_{Q^j}(u_i)^{-p}\mu(\{u_i\})=\sum_{u_i\in \supt \mu} \rho_{Q^0}(u_i)^{-p}\mu(\{u_i\}).
\end{equation}
Since $\rho_{Q^j}(u_{i_0})=1$, we have $\sum_{u_i\in \supt \mu} \rho_{Q^j}(u_i)^{-p}\mu(\{u_i\})>0$. Hence, 
\begin{equation}
\label{eq local 1003}
\lim_{j\rightarrow \infty} \log\sum_{u_i\in \supt \mu} \rho_{Q^j}(u_i)^{-p}\mu(\{u_i\})=\log\sum_{u_i\in \supt \mu} \rho_{Q^0}(u_i)^{-p}\mu(\{u_i\}).
\end{equation}

Equations \ref{eq local 1004} and \eqref{eq local 1003} imply that 
\begin{equation}
\lim_{j\rightarrow \infty}\Phi(Q^j)=\Phi(Q^0).
\end{equation}
\end{proof}

Let $S$ be a $k$-dimensional subspace of $\mathbb{R}^n$. Write $v\in S^{n-1}$ as 
\begin{equation}
\label{eq local 1005}
v=(v_2\cos \phi, v_1 \sin \phi),
\end{equation}
where $v_2\in S^{k-1}\subset S$, $v_1 \in S^{n-k-1}\subset S^\perp$ and $0\leq \phi\leq\pi/2$.

\begin{lemma}
\label{lemma function lower bound}
Suppose $u_1, \dots, u_N$ are $N$ unit vectors such that they are not concentrated on any great subspheres. Let $f:S^{n-1}\rightarrow \mathbb{R}$ be defined as 
\begin{equation}
f(v)=\max_{s+1\leq i\leq N}|v\cdot u_i|,
\end{equation}
where $1\leq s\leq N-1$ is such that $S=\Span\{u_1,\dots, u_s\}$ is a proper subspace of $\mathbb{R}^n$. Then, there exists constants $0<c<1$ and $0<\delta_0<\pi/2$ such that
\begin{equation}
f(v)\geq c,
\end{equation}
for each $v\in S^{n-1}$ with $\phi> \pi/2-\delta_0$. Here $\phi$ comes from the general polar coordinate expression \eqref{eq local 1005}.
\end{lemma}
\begin{proof}
Note that $f$ is uniformly continuous on $S^{n-1}$. Since $u_1,\dots, u_N$ are not concentrated on any great subspheres, 
\begin{equation}
f(v)>0, \qquad v\in S^\perp.
\end{equation}
By continuity of $f$, there exists $0<c_1<1$ such that 
\begin{equation}
\label{eq local 1006}
f(v)>c_1, \qquad v\in S^\perp.
\end{equation}
Moreover, since $f$ is uniformly continuous, there exists sufficiently small $0<\delta_1<1$ such that $||v_1-v_2||<\delta_1$ implies $|f(v_1)-f(v_2)|<\frac{c_1}{2}$. This, when combined with \eqref{eq local 1006}, show that there exists $0<c<1$ such that 
\begin{equation}
f(v)\geq c, \qquad \text{for } v\in S^{n-1} \text{ with } \text{dist}(v,S^\perp)<\delta_1.
\end{equation}
The desired result now follows from the fact that we can find a sufficiently small $0<\delta_0<1$ such that if $v\in S^{n-1}$ is such that $\phi>\frac{\pi}{2}-\delta_0$, then $\text{dist}(v,S^\perp)<\delta_1$.
\end{proof}

The following lemma partitions $S^{n-1}$ according to the support of a given measure $\mu$.
\begin{lemma}
\label{lemma partition}
Suppose $u_1, \dots, u_N$ are $N$ unit vectors such that they are not concentrated on any great subsphere. Let $1\leq s\leq N-1$ be such that $S=\Span \{u_1,\dots, u_s\}$ is a proper subspace of $\mathbb{R}^n$. Let $R\geq 1$. For $1\leq \rho_1,\dots, \rho_s\leq R$ and sufficiently small $0<t<1$ such that $\arccos \frac{ct}{R}>\frac{\pi}{2}-\delta_0$ where $c$ and $\delta_0$ come from Lemma \ref{lemma function lower bound}, let
\begin{equation}
K^t = \text{conv}\{\pm \rho_1 u_1,\dots, \pm \rho_s u_s, \pm tu_{s+1}, \dots, \pm t u_N\}.
\end{equation}
Denote 
\begin{align*}
\Omega_1 &=\left\{v\in S^{n-1}: \arccos \frac{ct}{R}<\phi<\frac{\pi}{2}\right\}\\
\Omega_2 &= \left\{v\in S^{n-1}: 0\leq \phi<\arccos t\right\}\\
\Omega_3 &= \left\{v\in S^{n-1}: \arccos t\leq \phi \leq \arccos \frac{ct}{R}\right\}.
\end{align*}
Then, for $v\in \Omega_1$,
\begin{equation}
h_{K^t}(v)\leq t, \qquad\text{and }\qquad h_{K^0}(v)\geq \cos \phi;
\end{equation}
for $v\in \Omega_2$,
\begin{equation}
h_{K^t}(v)=h_{K^0}(v);
\end{equation}
for $v\in \Omega_3$,
\begin{equation}
h_{K^t}(v)\leq R, \qquad \text{and } \qquad h_{K^0}(v)\geq \cos \phi.
\end{equation}
\end{lemma}
\begin{proof}
Throughout the proof, we will use the general polar coordinates \eqref{eq local 1005}.

Assume $v\in \Omega_1$. For $i=1,\dots, s$, 
\begin{equation}
\rho_i|u_i\cdot v|=\rho_i \cos \phi |u_i\cdot v_2|\leq R\cos \phi \leq R\cdot \frac{ct}{R}\leq t. 
\end{equation}
Since 
\begin{equation}
h_{K^t}(v)=\max \left\{\max_{i=1, \dots, s}\rho_i|u_i\cdot v|, \max_{i=s+1,\dots, N}t |u_i\cdot v|\right\},
\end{equation}
we have $h_{K^t}(v)\leq t$. On the other side, since $\rho_i\geq 1$, we have $B_S\subset K^0$ where $B_S$ is the unit ball in $S$. Hence,
\begin{equation}
h_{K^0}(v)\geq h_{B_S}(v)=\cos \phi.
\end{equation}

Assume now, $v\in \Omega_2$. By the definition of support function and that $B_S\subset K^t$,
\begin{equation}
h_{K^t}(v)\geq h_{B_S}(v)=\cos \phi>t.
\end{equation}
Since $t|v\cdot u|\leq t$ for any $u\in S^{n-1}$, we have
\begin{equation}
h_{K^t}(v)=\max \left\{\max_{i=1, \dots, s}\rho_i|u_i\cdot v|, \max_{i=s+1,\dots, N}t |u_i\cdot v|\right\}=\max_{i=1, \dots, s}\rho_i|u_i\cdot v|=h_{K^0}(v).
\end{equation}

Finally, let us assume $v\in \Omega_3$. By the fact that $\rho_{K^t}\leq R$, it is apparent that $h_{K^t}(v)\leq R$. The fact that $h_{K^0}(v)\geq \cos \phi$ follows from that $B_S\subset K^0$.
\end{proof}

The following lemma solves the optimization problem \eqref{eq variational problem}.
\begin{lemma}
\label{lemma existence of maximizer}
Let $-1<p<0$ and $\mu$ be an even discrete measure on $S^{n-1}$ whose support is not contained in any great subspheres. Then, there exists $K\in \mathcal{K}_e^n$ such that 
\begin{equation}
\Phi(K)=\sup_{Q\in D_\mu} \Phi(Q),
\end{equation}
where $\Phi$ is as defined in \eqref{eq local 1002}.
\end{lemma}
\begin{proof}
Suppose $\supt \mu  =\{\pm u_1, \pm u_2, \dots, \pm u_N\}$ and 
\begin{equation*}
\mu = \sum_{i=1}^N \mu_i \left(\delta_{u_i}+\delta_{-u_i}\right).
\end{equation*}

Suppose $Q^j\subset D_\mu$ is a maximization sequence. Let 
\begin{equation*}
\rho_i^j=\rho_{Q^j}(u^i).
\end{equation*}

Since $\Phi$ is homogeneous of degree 0, we may rescale $Q^j$ and assume $\max_{i}\rho_{i}^j=1$. By Blaschke's selection theorem, after possibly taking a subsequence, we may assume that there exists an origin-symmetric compact convex set $Q^0$ such that $Q^j$ converges to $Q^0$ in Hausdorff metric. Moreover,  
\begin{equation}
\rho_i^0:= \rho_{Q^0}(u_i) = \lim_{j\rightarrow \infty} \rho_i^j.
\end{equation}

By Lemma \ref{lemma convergence}, after possibly taking another subsequence, we may assume 
\begin{equation}
\label{eq local 1014}
\Phi(Q^0)=\lim_{j\rightarrow \infty}\Phi(Q^j)=\sup_{Q\in D_\mu} \Phi(Q).
\end{equation}

It remains to show that $o\in \text{int} Q^0$. 

We argue by contradiction and assume that there exists a proper subspace $S$ of $\mathbb{R}^n$ such that $Q^0\subset S$ and $\Span Q^0=S$. Let $k=\dim S$. Since $S$ is a proper subspace, we have $1\leq k<n$. By relabelling, we may assume there exists $1\leq s <N$ such that 
\begin{equation}
\pm u_1,\dots, \pm u_s \in S,
\end{equation}
and 
\begin{equation}
\pm u_{s+1},\dots, \pm u_N\notin S.
\end{equation}

Utilizing the fact that $\Phi$ is homogeneous of degree 0 again and that $\Span Q^0=S$, we may rescale $Q^0$ so that there exists $R\geq 1$ such that 
\begin{equation}
1\leq \rho_1^0,\dots, \rho_s^0\leq R,
\end{equation}
and
\begin{equation}
\rho_{s+1}^0,\dots, \rho_{N}^0=0.
\end{equation}

For sufficiently small $0<t<1$ such that $\arccos \frac{ct}{R}>\frac{\pi}{2}-\delta_0$ where $c$ and $\delta_0$ come from Lemma \ref{lemma function lower bound}, let 
\begin{equation}
K^t = \text{conv}\{\pm \rho_1 u_1,\dots, \pm \rho_s u_s, \pm tu_{s+1}, \dots, \pm t u_N\}.
\end{equation}
Note that $K_t \in D_\mu$. We are going to reach the desired contradiction by showing that for some $t$, $\Phi(K^t)>\Phi(Q^0)$.

Towards this end, for each $K\in \mathcal{K}_e^n$, write 
\begin{equation}
\mathcal{N}_p(K)=-\frac{1}{p}\log \sum_{u_i\in \supt \mu} \rho_{K}(u_i)^{-p}\mu_i = -\frac{1}{p}\log \left(2\sum_{i=1}^N\rho_{K}(u_i)^{-p}\mu_i\right).
\end{equation}
Let $\Delta_1(t) = \frac{1}{n\omega_n}\mathcal{E}(K^t)-\frac{1}{n\omega_n}\mathcal{E}(Q^0)$ and  $\Delta_2(t)=\mathcal{N}_p(K^t)-\mathcal{N}_p(Q^0)$. Note that 
\begin{equation}
\label{eq local 1013}
\Phi(K^t)-\Phi(Q^0)=\Delta_1(t)+\Delta_2(t).
\end{equation}

By Lemma \ref{lemma partition} and notice that $\Omega_1, \Omega_2, \Omega_3$ is a partition of $S^{n-1}$, 
\begin{equation}
\label{eq local 1009}
\begin{aligned}
n\omega_n\Delta_1(t)&\geq \left[-\int_{\Omega_1} \log tdv +\int_{\Omega_1}\log \cos \phi dv\right]+\left[-\int_{\Omega_3} \log Rdv+\int_{\Omega_3}\log \cos \phi dv\right]\\
&=-k\omega_k(n-k)\omega_{n-k}\log t \int_{\arccos \frac{ct}{R}}^\frac{\pi}{2} \cos^{k-1}\phi \sin^{n-k-1}\phi d\phi \\
&\phantom{we}- k\omega_k(n-k)\omega_{n-k}\log R\int_{\arccos t}^{\arccos \frac{ct}{R}} \cos^{k-1}\phi \sin^{n-k-1}\phi d\phi\\
& \phantom{we}+k\omega_k(n-k)\omega_{n-k}\int_{\arccos t}^{\pi/2}\log(\cos\phi)\cos^{k-1}\phi \sin^{n-k-1}\phi d\phi\\
&\geq -k\omega_k(n-k)\omega_{n-k}\log t \int_{\arccos \frac{ct}{R}}^\frac{\pi}{2} \cos^{k-1}\phi \sin^{n-k-1}\phi d\phi\\
&\phantom{we}- k\omega_k(n-k)\omega_{n-k}\log R\left(\arccos \frac{ct}{R}-\arccos t\right)\\
& \phantom{we}+k\omega_k(n-k)\omega_{n-k}\int_{\arccos t}^{\pi/2}\log(\cos\phi)\cos^{k-1}\phi \sin^{n-k-1}\phi d\phi\\
&=:k\omega_k(n-k)\omega_{n-k}g_1(t).
\end{aligned}
\end{equation}

By the definition of $K^t$ and $\Delta_2(t)$,
\begin{equation}
\label{eq local 1007}
\begin{aligned}
\Delta_2(t)&\geq -\frac{1}{p}\log \left(2\sum_{i=1}^s\left(\rho_i^0\right)^{-p}\mu_i+2\sum_{i=s+1}^N t^{-p}\mu_i\right)+\frac{1}{p}\log \left(2\sum_{i=1}^s\left(\rho_i^0\right)^{-p}\mu_i\right)\\
&=-\frac{1}{p}\log \frac{\sum_{i=1}^s\left(\rho_i^0\right)^{-p}\mu_i+\sum_{i=s+1}^N \mu_i t^{-p}}{\sum_{i=1}^s\left(\rho_i^0\right)^{-p}\mu_i}\\
&= -\frac{1}{p}\log \frac{a+bt^{-p}}{a}\\
&=:g_2(t),
\end{aligned}
\end{equation}
where $a=\sum_{i=1}^s\left(\rho_i^0\right)^{-p}\mu_i>0$ and $b=\sum_{i=s+1}^N \mu_i>0$.

Note that $\lim_{t\rightarrow 0^+}g_1(t)=0$. Indeed, for sufficiently small $t$,
\begin{equation}
\label{eq local 1008}
\begin{aligned}
|g_1(t)|&\leq |\log t|(\frac{\pi}{2}-\arccos\frac{ct}{R})+\log R (\arccos \frac{ct}{R}-\arccos t)+\int_{\arccos t}^\frac{\pi}{2} |\log \cos\phi| d\phi\\
&\leq |\log t| \arcsin \frac{ct}{R}+\log R (\arcsin\frac{ct}{R}-\arcsin t)+\int_{0}^t\frac{1}{\sqrt{1-x^2}}|\log x|dx\\
&\leq |\log t| \arcsin \frac{ct}{R}+\log R (\arcsin\frac{ct}{R}-\arcsin t)+2\int_{0}^t|\log x|dx\\
&\longrightarrow 0,
\end{aligned}
\end{equation}
as $t\rightarrow 0$. Also, it is straightforward to see that $\lim_{t\rightarrow 0^+}g_2(t)=0$.

Let $G(t)=\frac{k\omega_k(n-k)\omega_{n-k}}{n\omega_n}g_1(t)+g_2(t)$. From \eqref{eq local 1009}, \eqref{eq local 1007}, and \eqref{eq local 1008}, we see that 
\begin{equation}
\label{eq local 1012}
\Delta_1(t)+\Delta_2(t)\geq G(t),
\end{equation}
and 
\begin{equation}
\label{eq local 1011}
\lim_{t\rightarrow 0^+}G(t)=0.
\end{equation}

By direct computation, for sufficiently small $t>0$,
\begin{equation}
\begin{aligned}
g_1'(t)=&-\frac{1}{t}\int_{\arccos \frac{ct}{R}}^\frac{\pi}{2} \cos^{k-1}\phi\sin^{n-k-1}\phi d\phi\\
&\phantom{we}-\log t \cos^{k-1}\left(\arccos \frac{ct}{R}\right)\sin^{n-k-1}\left(\arccos \frac{ct}{R}\right)\frac{1}{\sqrt{1-\left(\frac{ct}{R}\right)^2}}\frac{c}{R}\\
&\phantom{we}-\log R \left(-\frac{1}{\sqrt{1-\left(\frac{ct}{R}\right)^2}}\frac{c}{R}+\frac{1}{\sqrt{1-t^2}}\right)\\
&\phantom{we} +\log\left(\cos(\arccos t)\right)\cos^{k-1}(\arccos t)\sin^{n-k-1}(\arccos t)\frac{1}{\sqrt{1-t^2}}\\
&\geq -\frac{\arcsin \frac{ct}{R}}{t}-\log R \left(-\frac{1}{\sqrt{1-\left(\frac{ct}{R}\right)^2}}\frac{c}{R}+\frac{1}{\sqrt{1-t^2}}\right)\\
&\phantom{we} + \log t\left[ -\cos^{k-1}\left(\arccos \frac{ct}{R}\right)\sin^{n-k-1}\left(\arccos \frac{ct}{R}\right)\frac{1}{\sqrt{1-\left(\frac{ct}{R}\right)^2}}\frac{c}{R}\right.\\
&\phantom{wwrewrew}+\cos^{k-1}(\arccos t)\sin^{n-k-1}(\arccos t)\frac{1}{\sqrt{1-t^2}}\Bigg]\\
&=: C(t)+\log t\cdot D(t),
\end{aligned}
\end{equation}
where $C(t)$ and $D(t)$ are bounded terms when $t>0$ is sufficiently small.

On the other side, for $t>0$ sufficiently small,
\begin{equation}
g_2'(t)=\frac{b}{a+t^{-p}b}t^{-p-1}\geq \frac{b}{2a}t^{-p-1}.
\end{equation}

Hence, 
\begin{equation}
\label{eq local 1010}
G'(t)\geq \frac{k\omega_k(n-k)\omega_{n-k}}{n\omega_n}C(t)+\frac{k\omega_k(n-k)\omega_{n-k}}{n\omega_n}D(t)\log t+\frac{b}{2a}t^{-p-1}.
\end{equation}
Since $-1<p<0$, when $t>0$ is sufficiently small, the right side of \eqref{eq local 1010} is positive. Hence, there exists $\delta_1>0$ such that for each $t\in (0,\delta_1)$
\begin{equation}
G'(t)>0.
\end{equation}
This, combined with \eqref{eq local 1011}, implies that there exists $t_0>0$ such that 
\begin{equation}
G(t_0)>0.
\end{equation}
By \eqref{eq local 1012}, this implies that
\begin{equation}
\Delta_1(t_0)+\Delta_2(t_0)>0.
\end{equation}
By \eqref{eq local 1013},
\begin{equation}
\Phi(K^{t_0})>\Phi(Q^0).
\end{equation}
But, this is a contradiction to \eqref{eq local 1014}.
\end{proof}

Lemmas \ref{lemma optimization problem} and \ref{lemma existence of maximizer} immediate imply:
\begin{theorem}
Suppose $p\in (-1,0)$ and $\mu$ is a non-zero, even, discrete, finite, Borel measure on $S^{n-1}$. There exists an origin-symmetric polytope $K\in \mathcal{K}_e^n$ such that $\mu=J_p(K,\cdot)$ if and only if $\mu$ is not concentrated entirely on any great subspheres.
\end{theorem}
\begin{proof}
The only if part is obvious while the if part follows from Lemmas \ref{lemma optimization problem} and \ref{lemma existence of maximizer}.
\end{proof}
\bibliography{mybib}{} 

\begin{thebibliography}{10}

\bibitem{MR0007625}
A.~D. Aleksandrov.
\newblock Existence and uniqueness of a convex surface with a given integral
  curvature.
\newblock {\em C. R. (Doklady) Acad. Sci. URSS (N.S.)}, 35:131--134, 1942.

\bibitem{MR1714339}
B.~Andrews.
\newblock Gauss curvature flow: the fate of the rolling stones.
\newblock {\em Invent. Math.}, 138:151--161, 1999.

\bibitem{MR1949167}
B.~Andrews.
\newblock Classification of limiting shapes for isotropic curve flows.
\newblock {\em J. Amer. Math. Soc. (JAMS)}, 16:443--459 (electronic), 2003.

\bibitem{MR2123199}
F.~Barthe, O.~Gu{\'e}don, S.~Mendelson, and A.~Naor.
\newblock A probabilistic approach to the geometry of the {$l^n_p$}-ball.
\newblock {\em Ann. Probab.}, 33:480--513, 2005.

\bibitem{MR3523119}
J.~Bertrand.
\newblock Prescription of {G}auss curvature using optimal mass transport.
\newblock {\em Geom. Dedicata}, 183:81--99, 2016.

\bibitem{BF}
K.~J. B{\"o}r{\"o}czky and F.~Fodor.
\newblock The ${L}_p$ dual {M}inkowski problem for $p>1$ and $q>0$.
\newblock {\em preprint}.

\bibitem{Boroczky20062015}
K.~J. B\"{o}r\"{o}czky, P.~Heged\H{u}s, and G.~Zhu.
\newblock On the discrete logarithmic minkowski problem.
\newblock {\em Int. Math. Res. Not. (IMRN)}, 2015.

\bibitem{MR3415694}
K.~J. B{\"o}r{\"o}czky and M.~Henk.
\newblock Cone-volume measure of general centered convex bodies.
\newblock {\em Adv. Math.}, 286:703--721, 2016.

\bibitem{BHP}
K.~J. B\"{o}r\"{o}czky, M.~Henk, and H.~Pollehn.
\newblock Subspace concentration of dual curvature measures of symmetric convex
  bodies.
\newblock {\em J. Differential Geom.}, in press.

\bibitem{MR2964630}
K.~J. B{\"o}r{\"o}czky, E.~Lutwak, D.~Yang, and G.~Zhang.
\newblock The log-{B}runn-{M}inkowski inequality.
\newblock {\em Adv. Math.}, 231:1974--1997, 2012.

\bibitem{BLYZ}
K.~J. B{\"o}r{\"o}czky, E.~Lutwak, D.~Yang, and G.~Zhang.
\newblock The logarithmic {M}inkowski problem.
\newblock {\em J. Amer. Math. Soc. (JAMS)}, 26:831--852, 2013.

\bibitem{MR3316972}
K.~J. B{\"o}r{\"o}czky, E.~Lutwak, D.~Yang, and G.~Zhang.
\newblock Affine images of isotropic measures.
\newblock {\em J. Differential Geom.}, 99:407--442, 2015.

\bibitem{BLYZZ}
K.~J. B{\"o}r{\"o}czky, E.~Lutwak, D.~Yang, G.~Zhang, and Y.~Zhao.
\newblock The dual {M}inkowski problem for symmetric convex bodies.
\newblock {\em preprint}.

\bibitem{MR3680945}
S.~Chen, Q.-R. Li, and G.~Zhu.
\newblock On the {$L_p$} {M}onge-{A}mp\`ere equation.
\newblock {\em J. Differential Equations}, 263(8):4997--5011, 2017.

\bibitem{CLZ}
S.~Chen, Q.-R. Li, and G.~Zhu.
\newblock The logarithmic minkowski problem for non-symmetric measures.
\newblock {\em Trans. Amer. Math. Soc.}, accepted.

\bibitem{MR2204749}
W.~Chen.
\newblock {$L_p$} {M}inkowski problem with not necessarily positive data.
\newblock {\em Adv. Math.}, 201:77--89, 2006.

\bibitem{MR0423267}
S.~Y. Cheng and S.~T. Yau.
\newblock On the regularity of the solution of the {$n$}-dimensional
  {M}inkowski problem.
\newblock {\em Comm. Pure Appl. Math.}, 29:495--516, 1976.

\bibitem{MR2254308}
K.-S. Chou and X.-J. Wang.
\newblock The {$L_p$}-{M}inkowski problem and the {M}inkowski problem in
  centroaffine geometry.
\newblock {\em Adv. Math.}, 205:33--83, 2006.

\bibitem{GHWXY}
R.~Gardner, D.~Hug, W.~Weil, S.~Xing, and D.~Ye.
\newblock General volumes in the {O}rlicz-{B}runn-{M}inkowski theory and a
  related minkowski problem i.
\newblock {\em preprint}.

\bibitem{MR2530600}
C.~Haberl and F.~E. Schuster.
\newblock Asymmetric affine {$L_p$} {S}obolev inequalities.
\newblock {\em J. Funct. Anal.}, 257:641--658, 2009.

\bibitem{MR3148545}
M.~Henk and E.~Linke.
\newblock Cone-volume measures of polytopes.
\newblock {\em Adv. Math.}, 253:50--62, 2014.

\bibitem{MR3725875}
M.~Henk and H.~Pollehn.
\newblock Necessary subspace concentration conditions for the even dual
  {M}inkowski problem.
\newblock {\em Adv. Math.}, 323:114--141, 2018.

\bibitem{MR3366857}
Y.~Huang, J.~Liu, and L.~Xu.
\newblock On the uniqueness of {$L_p$}-{M}inkowski problems: the constant
  {$p$}-curvature case in {$\Bbb{R}^3$}.
\newblock {\em Adv. Math.}, 281:906--927, 2015.

\bibitem{HLYZ}
Y.~Huang, E.~Lutwak, D.~Yang, and G.~Zhang.
\newblock Geometric measures in the dual {B}runn-{M}inkowski theory and their
  associated {M}inkowski problems.
\newblock {\em Acta Math.}, 216(2):325--388, 2016.

\bibitem{HLYZ2}
Y.~Huang, E.~Lutwak, D.~Yang, and G.~Zhang.
\newblock The ${L}_p$-{A}leksandrov problem for ${L}_p$-integral curvature.
\newblock {\em J. Differential Geom.}, in press.

\bibitem{HZ}
Y.~Huang and Y.~Zhao.
\newblock On the ${L}_p$-dual {M}inkowski problem.
\newblock {\em preprint}.

\bibitem{MR2132298}
D.~Hug, E.~Lutwak, D.~Yang, and G.~Zhang.
\newblock On the {$L_p$} {M}inkowski problem for polytopes.
\newblock {\em Discrete Comput. Geom.}, 33:699--715, 2005.

\bibitem{MR3366854}
H.~Jian, J.~Lu, and X.-J. Wang.
\newblock Nonuniqueness of solutions to the {$L_p$}-{M}inkowski problem.
\newblock {\em Adv. Math.}, 281:845--856, 2015.

\bibitem{MR3479715}
H.~Jian, J.~Lu, and G.~Zhu.
\newblock Mirror symmetric solutions to the centro-affine {M}inkowski problem.
\newblock {\em Calc. Var. Partial Differential Equations}, 55:41, 2016.

\bibitem{LSW}
Q.R. Li, W.M. Sheng, and X.J. Wang.
\newblock Flow by {G}auss curvature to the {A}leksandrov and dual {M}inkowski
  problems.
\newblock {\em J. Eur. Math. Soc.}, in press.

\bibitem{MR2997361}
J.~Lu and X.-J. Wang.
\newblock Rotationally symmetric solutions to the {$L_p$}-{M}inkowski problem.
\newblock {\em J. Differential Equations}, 254:983--1005, 2013.

\bibitem{MR2652209}
M.~Ludwig.
\newblock General affine surface areas.
\newblock {\em Adv. Math.}, 224:2346--2360, 2010.

\bibitem{MR2680490}
M.~Ludwig and M.~Reitzner.
\newblock A classification of {${\rm SL}(n)$} invariant valuations.
\newblock {\em Ann. of Math. (2)}, 172:1219--1267, 2010.

\bibitem{MR1231704}
E.~Lutwak.
\newblock The {B}runn-{M}inkowski-{F}irey theory. {I}. {M}ixed volumes and the
  {M}inkowski problem.
\newblock {\em J. Differential Geom.}, 38:131--150, 1993.

\bibitem{MR1378681}
E.~Lutwak.
\newblock The {B}runn-{M}inkowski-{F}irey theory. {II}. {A}ffine and geominimal
  surface areas.
\newblock {\em Adv. Math.}, 118:244--294, 1996.

\bibitem{MR1316557}
E.~Lutwak and V.~Oliker.
\newblock On the regularity of solutions to a generalization of the {M}inkowski
  problem.
\newblock {\em J. Differential Geom.}, 41:227--246, 1995.

\bibitem{MR1863023}
E.~Lutwak, D.~Yang, and G.~Zhang.
\newblock {$L_p$} affine isoperimetric inequalities.
\newblock {\em J. Differential Geom.}, 56:111--132, 2000.

\bibitem{MR1987375}
E.~Lutwak, D.~Yang, and G.~Zhang.
\newblock Sharp affine {$L_p$} {S}obolev inequalities.
\newblock {\em J. Differential Geom.}, 62:17--38, 2002.

\bibitem{MR2067123}
E.~Lutwak, D.~Yang, and G.~Zhang.
\newblock On the {$L_p$}-{M}inkowski problem.
\newblock {\em Trans. Amer. Math. Soc.}, 356:4359--4370, 2004.

\bibitem{LYZpDC}
E.~Lutwak, D.~Yang, and G.~Zhang.
\newblock ${L}_p$ dual curvature measures.
\newblock {\em Adv. Math.}, 329:85 --132, 2018.

\bibitem{MR2262841}
A.~Naor.
\newblock The surface measure and cone measure on the sphere of {$l_p^n$}.
\newblock {\em Trans. Amer. Math. Soc.}, 359:1045--1079 (electronic), 2007.

\bibitem{MR1962135}
A.~Naor and D.~Romik.
\newblock Projecting the surface measure of the sphere of {$l_p^n$}.
\newblock {\em Ann. Inst. H. Poincar\'e Probab. Statist.}, 39:241--261, 2003.

\bibitem{MR2332603}
V.~Oliker.
\newblock Embedding {$\bold S^n$} into {$\bold R^{n+1}$} with given integral
  {G}auss curvature and optimal mass transport on {$\bold S^n$}.
\newblock {\em Adv. Math.}, 213:600--620, 2007.

\bibitem{MR2880241}
G.~Paouris and E.~Werner.
\newblock Relative entropy of cone measures and {$L_p$} centroid bodies.
\newblock {\em Proc. Lond. Math. Soc. (3)}, 104:253--286, 2012.

\bibitem{schneider2014}
R.~Schneider.
\newblock {\em Convex bodies: the {B}runn-{M}inkowski theory}, volume 151 of
  {\em Encyclopedia of Mathematics and its Applications}.
\newblock Cambridge University Press, Cambridge, expanded edition, 2014.

\bibitem{MR1901250}
A.~Stancu.
\newblock The discrete planar {$L_0$}-{M}inkowski problem.
\newblock {\em Adv. Math.}, 167:160--174, 2002.

\bibitem{MR2019226}
A.~Stancu.
\newblock On the number of solutions to the discrete two-dimensional
  {$L_0$}-{M}inkowski problem.
\newblock {\em Adv. Math.}, 180:290--323, 2003.

\bibitem{MR2729006}
G.~Xiong.
\newblock Extremum problems for the cone volume functional of convex polytopes.
\newblock {\em Adv. Math.}, 225:3214--3228, 2010.

\bibitem{Zha2}
Y.~Zhao.
\newblock The dual {M}inkowski problem for negative indices.
\newblock {\em Calc. Var. Partial Differential Equations}, 56(2):Art. 18, 16,
  2017.

\bibitem{Zha1}
Y.~Zhao.
\newblock Existence of solutions to the even dual {M}inkowski problem.
\newblock {\em J. Differential Geom.}, in press.

\bibitem{MR3228445}
G.~Zhu.
\newblock The logarithmic {M}inkowski problem for polytopes.
\newblock {\em Adv. Math.}, 262:909--931, 2014.

\bibitem{MR3356071}
G.~Zhu.
\newblock The centro-affine {M}inkowski problem for polytopes.
\newblock {\em J. Differential Geom.}, 101:159--174, 2015.

\bibitem{MR3352764}
G.~Zhu.
\newblock The {$L_p$} {M}inkowski problem for polytopes for {$0<p<1$}.
\newblock {\em J. Funct. Anal.}, 269:1070--1094, 2015.

\bibitem{guangxianindiana}
G.~Zhu.
\newblock The {$L_p$} {M}inkowski problem for polytopes for {$p<0$}.
\newblock {\em Indiana Univ. Math. J.}, 66(4):1333--1350, 2017.

\bibitem{MR3255458}
D.~Zou and G.~Xiong.
\newblock Orlicz-{J}ohn ellipsoids.
\newblock {\em Adv. Math.}, 265:132--168, 2014.

\end{thebibliography}
\bibliographystyle{plain} \end{document}